\documentclass[a4paper]{amsart}
\usepackage{amssymb,amsfonts}
\usepackage[all,arc]{xy}
\usepackage{enumerate}
\usepackage{mathrsfs}
\usepackage{mathtools}
\usepackage{tikz}
\usepackage{pgfplots}
\usetikzlibrary{calc}
\usetikzlibrary{arrows,snakes,backgrounds}
\usepackage{color}   
\usepackage{marginnote}
\usepackage{tikz-cd}
\usetikzlibrary{positioning}

\usepackage{enumitem}
\usepackage{hyperref}
\usepackage{MnSymbol} 
\hypersetup{
    colorlinks=true, 
    linktoc=all,     
    citecolor=black,
    filecolor=black,
    linkcolor=blue,
    urlcolor=black
}
\usepackage{cleveref}
\usetikzlibrary{arrows}

\usepackage[normalem]{ulem}
\usepackage{marginnote}

\newtheorem{thm}{Theorem}[section]
\newtheorem{cor}[thm]{Corollary}
\newtheorem{prop}[thm]{Proposition}
\newtheorem{lem}[thm]{Lemma}

\newtheorem*{claim*}{Claim}
\newtheorem{problem}[thm]{Problem}
\newtheorem*{openproblem*}{Problem}
\newtheorem*{quest*}{Question}
\newtheorem*{problem*}{Problem}

\theoremstyle{definition}
\newtheorem{defn}[thm]{Definition}

\theoremstyle{remark}
\newtheorem{rem}[thm]{Remark}

\newcommand{\bH}{\mathbb{H}}

\newcommand{\bN}{\mathbb{N}}

\newcommand{\bQ}{\mathbb{Q}}
\newcommand{\bR}{\mathbb{R}}

\newcommand{\bZ}{\mathbb{Z}}

\newcommand\Diff{\mathrm{Diff}}

\newcommand\Homeo{\mathrm{Homeo}}
\newcommand\BDiff{\mathrm{BDiff}}

\newcommand{\hcoker}{/\!\!/}

\newcommand{\tH}{\text{\textnormal{Homeo}}}
\newcommand{\BH}{\mathrm{B}\text{\textnormal{Homeo}}}

\renewcommand{\paragraph}[1]{\medskip \noindent {\bf #1.}}

\newcounter{notes}

\addtocontents{toc}{\protect\setcounter{tocdepth}{1}}
\makeatletter
\let\c@equation\c@thm
\makeatother
\numberwithin{equation}{section}
\title{On the bordism group for group actions on the torus}

\author{Kathryn Mann}
\address{Department of Mathematics\\
Cornell University\\
Ithaca, NY, 14850 }
\email{k.mann@cornell.edu}
\author{Sam Nariman}
\address{Department of Mathematics\\
  Purdue University\\
150 N. University Street\\
West Lafayette, IN 47907-2067\\
}
\email{snariman@purdue.edu}

\begin{document}

\begin{abstract}
In this short note, we study the bordism problem for group actions on the torus and give examples of groups acting on the torus by diffeomorphisms isotopic to the identity that cannot be extended to an action on a bounding 3-manifold.   This solves a question raised in the previous work of the authors.  
\end{abstract}
\maketitle

\section{Introduction}

In \cite{Browder}, Browder introduced the notion of (oriented) bordism for diffeomorphisms.  
Two orientation preserving diffeomorphisms $f_i: M_i \to M_i$ of closed, oriented $n$-manifolds $M_i$ are {\em bordant} if there is an oriented bordism $W$ between $M_1$ and $M_2$ and an orientation preserving diffeomorphism $H: W \to W$ that restricts to $f_i$ on $M_i$.  
Bordism classes of diffeomorphism groups on $n$-manifolds form an abelian group denoted $\Delta_n$ (sometimes written $\Delta_{n+}$ to emphasize orientation).  
Kreck (see \cite{Kreck}) computed these for $n>3$, shortly after this Melvin \cite{Melvin} showed $\Delta_3 = 0$, and the group $\Delta_2$ was eventually computed by Bonahon \cite{Bonahon}  using work of Scharelmann.  

Viewing a diffeomorphism as a $\bZ$-action, Browder's definition readily generalizes to other groups: 

\begin{defn} \label{bordism def}
For $n$-manifolds $M_1$ and $M_2$ and a discrete group $\Gamma$, two homomorphisms $\rho_1, \rho_2: \Gamma \to \Diff(M_i)$ are \emph{bordant} if there is a $(n+1)$-manifold $W$ and a representation $\phi: \Gamma \to \Diff(W)$ such that $\partial W = M_1 \sqcup -M_2$ and such that the restriction of $\phi(\gamma)$ to $M_i$ agrees with $\rho_i(\gamma)$ for each $\gamma \in \Gamma$.  
\end{defn}
Bordism classes of group actions on $n$-manifolds form an abelian group $\Delta(n, \Gamma)$ under disjoint union.  
As in Browder, we will assume going forward that all manifolds orientable and all diffeomorphisms are orientation preserving.  

Here we study the role that the algebraic structure of $\Gamma$ plays in the structure of the group $\Delta(n, \Gamma)$ rather than the role played by the topology of the manifold on which it acts.  To focus on this algebraic aspect, we wish to require of our actions of $\Gamma$ on $M_i$ to have the property that every individual element of $\Gamma$ extends to act on every manifold $W$ bounded by $M_1 \sqcup M_2$.   A simple way to ensure this is to take actions by isotopically trivial diffeomorphisms: requiring that $\rho_i(\Gamma)$ lies in the identity component $\Diff_0(M_i)$ of $\Diff(M_i)$, the action of each individual element may be smoothly isotoped to the identity along a small collar neighborhood of $M_i$ in $W$, defining a diffeomorphism of $W$ supported in a neighborhood of the boundary.    This parallels the framework set up by Ghys in \cite{Ghys} and studied further in previous work of the authors \cite{MannNariman} on obstructions to extending isotopically trivial actions on some manifold $M$ to a manifold $W$ with $\partial W = M$.  However, neither of these works furnishes an example of such a group $\Gamma$ such that $\Delta(n, \Gamma)$ is nontrivial for any $n >1$.  This is the primary motivation of this note.   We exhibit an example of such a group, and give both geometric-topological tools and cohomological tools to further pursue the bordism problem.  

\subsection*{Results} 
Let $G' \subset \Diff_0(S^1)$ be the smooth conjugate of the standard action of Thompson's group constructed by Ghys and Sergiescu in \cite{MR896095}.  The notation $G'$ is adopted from their work.   Since $G'$ is well known to be finitely generated, the following gives the first example of a finitely generated group acting on a 2-manifold with nontrivial bordism group.  
\begin{thm}\label{Thompson}
Let $\Gamma = G' \times G'$.  The product action $\rho: G' \times G' \to \Diff_0(S^1 \times S^1)$ is a nontrivial element of $\Delta_{(\Gamma,2)}$.
\end{thm}

As in \cite{MannNariman}, we give two independent approaches to this problem: one using geometric topology, and one using cohomology of diffeomorphism groups.  
The geometric approach relies heavily on torsion elements.  The key tool is the following theorem, its proof uses the geometrization theorem and an analysis of finite order diffeomorphisms of geometric manifolds.  

\begin{thm}\label{torsion}
Let $M$ be a $3$-manifold with $\partial M\cong T^2$, and suppose $G_1, G_2$ are subgroups of $\Diff_0(S^1)$ containing torsion elements of arbitrarily high order.  
 If the action of $G_1 \times G_2$ on $S^1 \times S^1 = T^2$ extends to a smooth action on $M$, then $M$ is diffeomorphic to a solid torus. 
\end{thm}
One can avoid the assumption that $M$ is orientable by lifting the action and running the same argument on the orientation cover (which will not affect the property that $G_i$ still has elements of arbitrarily high order), but for simplicity we continue to assume all manifolds orientable.  

The cohomological approach to the bordism problem uses powers of the Euler class as  cohomological obstructions to extending group actions. In this case, instead of assuming the existence of torsion, we assume a weaker condition of nonvanishing of powers of the Euler class as follows.  As is well known, $H^*(\BH_0(T^2);\bQ)\cong \bQ[x_1,x_2]$ where each $x_i$ is a cohomology classes of degree $2$.  One can think of $x_1$ and $x_2$ as the Euler classes from each factor of $T^2=S^1\times S^1$. 
We show the following.  
\begin{thm}\label{main1}
Let $M$ be an irreducible $3$-manifold whose boundary $\partial M$ is homeomorphic to a torus and suppose that the product action of  $H =  G_1\times G_2$  on $\partial M = S^1 \times S^1$ induces an injective map $\bQ[x_1,x_2]\to H^*(\mathrm{B}H;\bQ)$.    If this action extends to any $C^0$ action on $M$, then  $M\cong D^2\times S^1$. 
\end{thm}

If $G_1$ and $G_2$ satisfy the hypothesis of \Cref{torsion}, then $H= G_1\times G_2$ satisfies the hypothesis of \Cref{main1}. The reason is a torsion of order $k$ in each $G_i$ gives an embedding $\bZ/k\hookrightarrow \Diff_0(S^1)$. It is standard to see that for each $n$, the pullback of the $n$-th power of the universal Euler class to the group $H^{2n}(\mathrm{B}\bZ/k;\bZ)$ is nonzero and is a $k$-torsion. Therefore, if the pullback of the $n$-th power of the universal Euler class to $H^{2n}(G_i;\bZ)$ were a torsion element, its order would be divisible by $k$. Since  both $G_i$ have torsion elements of arbitrarily high order, the pullback of the powers of the Euler class are non-torsion classes. Therefore, when $M$ is irreducible, \Cref{torsion}  follows from \Cref{main1} but we also provide an approach using torsion elements which might be of an independent interest.

The proof of \Cref{main1} uses the solution of Kontsevich's conjecture for irreducible $3$-manifolds with non empty boundary by Hatcher and McCullough (\cite{MR1486644}).  The statement is a generalization of \cite[Proposition 2.2]{MannNariman}, which considered a more restrictive extension problem.   

\begin{rem}
The product action of $G' \times G'$ where $G'$ is Thompson's group, as discussed above, fits the conditions of Theorem \ref{main1}, so this gives an alternative proof of Theorem \ref{Thompson} in the restricted case where the bounding manifold $M$ is required to be irreducible.    
It would be interesting to see if one can drop the irreducibility condition in  \Cref{main1}.  
\end{rem}

Since our motivation comes from the question of bordism of group actions, we framed Theorem \ref{main1} as a rigidity statement for actions of a fixed group $H$.  However, if one is interested instead in finding obstructions to extending actions of a group to a {\em fixed} irreducible 3-manifold $M$, then our work applies to a much wider class of groups.   For a given irreducible $3$-manifold $M$ bounding $T^2$ (and not homeomorphic to the solid torus), we can find an integer $n(M)$ depending on $M$ such that $x_i^k\in H^{2k}(G_i;\bQ)$ cannot be nonzero for both $i=1,2$ when $k>n(M)$.    Thus,  one does not need injectivity of $\bQ[x_1,x_2]\to H^*(\mathrm{B}H;\bQ)$, but rather only that powers up to $n(M)$ do not vanish.  

%

The advantage of the cohomological approach is that 
\Cref{main1}, in principle, could provide an obstruction to extending actions of {\em torsion free} groups. For example, the mapping class group of a surface of genus $g$ with a marked point $\Gamma_{g,1}$ is a subgroup of $\tH_0(S^1)$. 
It is known (\cite{MR2871163}) that $e^{g-1}\in H^{2g-2}(\Gamma_{g,1};\bQ)$ is not zero where $e\in H^2(\Gamma_{g,1};\bQ)$ is the Euler class. On the hand, we know that $\Gamma_{g,1}$ has torsion free finite index subgroups (\cite[Theorem 6.9]{MR2850125}). Hence, we have torsion free subgroups of $\tH_0(S^1)$ that support high powers of the Euler class.  If we consider their diagonal embeddings in $\tH_0(S^1 \times S^1)$, we could use our method of proof for \Cref{main1} to obstruct extending such actions on torus to certain irreducible $3$-manifolds. 

\subsection*{Acknowledgment}SN was partially supported by  NSF grant DMS-1810644 and acknowledges the support from the European Research Council (ERC) under the European Union's Horizon 2020 research and innovation programme (grant agreement No. 682922).   KM was partially supported by NSF CAREER grant DMS 1844516 and a Sloan fellowship.   We also thank the referee for the helpful comments and suggestions. 

\section{Finite order diffeomorphisms of $3$-manifolds}
To prove \Cref{torsion}, we first analyze how existence of finite order diffeomorphisms of high order constrains the possible geometric structures on a $3$-manifold.  
This section can be read independently from the rest of the work.  

\begin{prop} \label{prop:isometry}
Let $M$ be a compact, orientable, irreducible 3-manifold, possibly with boundary.   
There exists $n=n(M)$ such that any nontrivial finite order diffeomorphism of $M$ of order at least $n$ acts on $M$ by isometries of a Seifert fibered geometric structure on $M$.  In particular, if $M$ is not itself geometric and Seifert fibered, then it does not admit finite order diffeomorphisms of arbitrarily high order. 
\end{prop} 

This proposition relies heavily on Thurston's geometrization conjecture, which we recall here: 
\begin{thm}[Geometrization] 
The interior of any compact, orientable 3-manifold can be split along a finite collection of essential,
pairwise disjoint, embedded spheres and tori into a canonical collection of finite-volume geometric
3-manifolds after capping off all boundary spheres by 3-balls.
\end{thm}
We refer the reader to \cite{BBBMP} for a  survey and further references.  
While we prove the general statement above, in our intended application $M$ is assumed to have nonempty boundary, and for this we need only use Thurston's geometrization theorem for Haken 3-manifolds.  

The proof of Proposition \ref{prop:isometry} proceeds by considering a decomposition of $M$ into geometric pieces (which we will eventually see is forced to be trivial if $M$ admits diffeomorphisms of arbitrarily high order).  Of the eight 3-dimensional geometries, the only two which are not Seifert fibered are $\mathbb{H}^3$ and Sol.  
As is well known, Mostow rigidity implies that the isometry group of a finite volume hyperbolic 3-manifold is finite, hence the maximal order of a finite order element is bounded.  This is also true of solvmanifolds, however we did not find a stand-alone proof in the literature, so provide one now.   

\begin{lem} \label{lem:sol}
Let $N$ be a finite volume solvmanifold.  There exists $k \in \bN$ such that $N$ has no finite order diffeomorphism of order greater than $k$.  
\end{lem}

\begin{proof}
Let $N = G/\Gamma$ be a finite volume solvmanifold, where $G = \mathrm{Sol}$ and $\Gamma$ is a discrete subgroup of $\mathrm{Isom}(\mathrm{Sol})$.  We recall some general structure theory, further details can be found in \cite{Scott}.  
The group $G$ has the structure of a split extension 
\[ 0 \to \mathbb{R}^2 \to G \to \mathbb{R} \to 0 \]
where $t \in \mathbb{R}$ acts on $\mathbb{R}^2$ by $t\cdot(x,y) = (e^t x, e^{-t}y)$.  Identifying $G$ with triples of real numbers $(x,y,t)$, the  planes $t$=constant give a foliation of $G$ invariant under isometries and $N$ is naturally a finite quotient of a torus bundle over the circle, hence compact.  
The identity component of $\mathrm{Isom}(G)$ is simply $G$ itself acting by left-multiplication and has index 8.  This implies that $N$ has a finite cover (of degree at most 8) that is the quotient of $G$ by a discrete subgroup.  Every such manifold is the mapping  torus of a linear Anosov map of $T^2$.   

Let $h$ be a finite order diffeomorphism of $N$.  By \cite[Thm. 8.2]{MeeksScott}, $h$ is conjugate to an isometry, so we assume without loss of generality that $h$ is an isometry.  It is also no loss of generality to lift $h$ to the mapping torus cover $N'$ of $N$ and prove that the order of the lift is bounded, so now we work with a finite order isometry of $N' = G/\Gamma'$, where $\Gamma' = \Gamma \cap G$.  Abusing notation, let $h$ denote this isometry, and suppose it is of order $d$.   Referring to the split extension sequence above, we have $\Gamma' \cap \bR^2 \cong \bZ \times \bZ$ and $\Gamma' \cap \bR \cong \bZ$, and if $\Gamma' \cap \bR$ is generated by $t \in \bR$, then the monodromy of the mapping torus is given by $\left( \begin{smallmatrix} e^t & 0 \\ 0 & e^{-t} \end{smallmatrix} \right)$, which in the basis given by the identification $\Gamma'  \cap \bR^2 \cong (\bZ \times \bZ)$ is some integer matrix $A \in \mathrm{SL}_2(\bZ)$.  
 
 Let $r$ be the largest root of $A$ in $\mathrm{SL}_2\mathbb{Z}$, i.e. the maximal number such that there exists some $B \in \mathrm{SL}_2\mathbb{Z}$ with $B^r = A$, such an $r$ exists since $\mathrm{SL}_2\mathbb{Z}$ is discrete.  The map $h$ lifts to an isometry $\hat{h}$ of $G$ preserving the vertical two-dimensional foliation of $G$.  Again, for simplicity, we can work instead with $\hat{h}^8$ which still descends to a finite order isometry of $N'$ (say of order $d'$) and now lies in $G$.  The subgroup generated by $\Gamma'$ and $\hat{h}$ is again discrete in $G$, with quotient the mapping torus of an Anosov map which is a root of $A$, since the $d'$-fold cover of this mapping torus is simply $N'$.  Thus $d' < r$, and since $8d' \geq d$, we conclude $d \leq 8r$.  
\end{proof}

We need a further lemma on Seifert fibered manifolds, which follows from work of Meeks and Scott.  

\begin{lem} \label{lem:seifert_torus_boundary}
Let $N$ be an irreducible Seifert fibered manifold with a torus boundary component $T$.  Assume $N$ is not diffeomorphic to the solid torus or an $S^1$ bundle over the annulus or Mobius band.   There exists $l\in \mathbb{N}$ such that any diffeomorphism of $N$ preserving $T$ of order at least $l$ has a nontrivial power which preserves the fibers of some Seifert fibration.   
\end{lem}

\begin{proof}
Fix a Seifert fibered manifold $N$ as in the statement, and suppose $f$ is a finite order diffeomorphism of $N$.  
By our assumptions on $N$,  it has at most two Seifert fiberings.  (With the exception of the twisted I-bundle over the torus and $I$-bundles over the Klein bottle, a manifold with torus boundary will in fact have a unique Seifert fibering, see \cite[1.1.2]{preaux2012survey} or \cite[\S 3]{Scott}).  
Thus, replacing $f$ with $f^2$, we conclude that $f$ preserves a Seifert fibration up to isotopy.  By \cite[Thm. 2.2]{MeeksScott}, $N$ therefore admits an $f^2$-invariant Seifert fibration.  Consider the induced action of $f^2$ on the base orbifold.  This is a finite order homeomorphism preserving a boundary component (corresponding to the boundary component $T$).  Unless this orbifold is the disc with one or zero cone points, the annulus, or the Mobius band (which are excluded by our assumptions), then it admits a singular hyperbolic structure and hence there is an upper bound, depending on the geometry of the orbifold, on the order of a finite order homeomorphism.  Call this bound $d$.  Thus, requiring that the order of $f$ be greater than $2d$ implies that some nontrivial power of $f$ preserves a Seifert fibration and acts trivially on the base orbifold, hence acts by rotating the fibers.  
\end{proof}

To deal with two adjacent Seifert fibered pieces that share a torus boundary in the JSJ decomposition, we need the following elementary lemma about invariant curves on tori. 
\begin{lem} \label{lem:invariant_curve}
Let $T = \bR^2 / \bZ^2$ be a torus, and let $a, b$ be simple closed curves representing a standard basis for homology.  
Suppose that $c$ is a simple closed curve invariant under the rotation $r: (x, y) \mapsto (x+ p/q, \, y)$, where $p/q \in \bQ$ is in lowest terms. 
Then $[c] \in H_1(T; \bZ)$ is of the form $[a]^k [b]^{nq}$ for some $k$ and $n$ in $\bZ$.  
\end{lem} 

\begin{proof}
Let $c$ be a curve as claimed.  Since $r$ acts freely on $T$ and $c$ is $r$-invariant, the $q$-fold covering map $T \to T/\langle r \rangle$ restricts to a $q$-fold covering map $c \to c/\langle r \rangle$.   
Take a lift $\tilde{c}$ of $c$ to $\bR^2$ based at 0, its endpoint is some point $(k, l) \in \bZ \times \bZ$.    Projecting $\tilde{c}$ to $\bR^2 / (\frac{1}{q}\bZ \times \bZ) = T/\langle r \rangle$ gives a closed curve that represents a $q$-fold cover of the circle $c/\langle r \rangle$.  In other words, there are $q$ distinct points of $\frac{1}{q}\bZ \times \bZ$ along $\tilde{c}$, all distinct from the origin and differing by translates by some $\mathbf{v} \in \frac{1}{q}\bZ \times \bZ$, where the endpoint of $\tilde{c}$ is the point $q \mathbf{v}$. We conclude that the second coordinate $l$ is a multiple of $q$.  
\end{proof}

With this groundwork in place, we can now prove the main Proposition.  

\begin{proof}[Proof of Proposition \ref{prop:isometry}]
Given $M$, let $j$ be the number of geometric pieces in a decomposition of $M$ by tori into finite volume geometric manifolds.  
Let $k$ be the maximum order of a finite order isometry of any hyperbolic or Sol geometry piece of $M$ (the latter only possibly occurring if $M$ is itself a solvmanifold, since finite volume solvmanifolds are compact), and set $k=1$ if $M$ has no such pieces.  
Let $l$ be larger than the product of the maximal orders of non-fiber preserving isometries of each Seifert fibered piece of $M$.  

Finally, in the case where $M$ has two adjacent Seifert fibered pieces in its JSJ decomposition, fix a torus $T$ separating the two pieces.  Let $a_1 \subset T$ be a regular fiber for the Seifert fibered structure of one of the pieces with boundary $T$, and fix a transverse simple curve $b_1$ so that $a_1, b_1$ form a basis for homology, giving an identification of $T$ with $\bR^2 / \bZ^2$ where rotating the fibers of this Seifert piece corresponds to rotating the first factor.  
Let $c$ be a fiber from the other adjacent piece, then $[c] = [a_1]^u [b_1]^v \in H_1(T; \bZ)$ and fix some $t \in \bN$ with $t > |v|$.  (This choice will allow us to later quote Lemma \ref{lem:invariant_curve}.)  In the case where these pieces have non-unique fiberings, hence finitely many, choose $t$ large enough to satisfy the property above over all choices of pairs of fiberings on the adjacent pieces.  

Suppose that $M$ admits some finite order diffeomorphism $f$ of order $m > tlkj!$.  
Then one may find a decomposition of $M$ into geometric pieces invariant under $f$ \cite{Zimmerman} hence $f^{j!}$ is a finite order diffeomorphism preserving each piece.  If $k>1$, then there is some invariant hyperbolic or sol geometry piece $M_0$ such that $f^{kj!}$ is the identity on $M_0$.  Since $f^{kj!}$ is finite order, it follows that it is the identity everywhere (using the easy fact that a finite order diffeomorphism cannot be the identity on an open set) hence $f$ is of order at most $lkj!$, contradicting our choice.  Thus, $k = 1$, and $M$ cannot itself be hyperbolic or a solvmanifold by Lemma \ref{lem:sol}, so $M$ has only Seifert fibered pieces.  

We now claim that $M$ in fact has only one geometric piece. To show this, suppose for contradiction that it had at least two pieces and consider two pieces with common boundary $T$ discussed above.   Since these pieces are noncompact, their geometric structure is either $\bH^2 \times \bR$ or $\widetilde{\mathrm{SL}}_2(\bR)$, so in particular they fit the conditions given in Lemma \ref{lem:seifert_torus_boundary}, using $f^{j!}$ as our finite order isometry.  
Thus, $f^{j!}$ has a nontrivial power, call this $g$, which is of order at least $t$ and preserves (fiberwise) the Seifert fibering of each piece. Since $g$ is a nontrivial finite order diffeomorphism, it acts by rotating the fibers on each piece.  Recall that fibers of distinct pieces are nonisotopic as curves on $T$, due to the minimality of the JSJ decomposition.  Let $a \subset T$ be a regular fiber of one piece and let $c \subset T$ denote a regular fiber of the other piece.  Both are $g$-invariant curves.  Since $c$ is not isotopic to $a$, taking a basis $\{ [a], [b] \}$ for homology as we chose in the set-up to this proof 
 we have $[c] = [a]^u[b]^v \in H_1(T, \bZ)$ for some $v \neq 0$.  Applying Lemma 2.5, we conclude that $v$ is a multiple of the order of $g$.  However,  by construction, the order of $g$ was chosen so that $|v| < t \leq |g|$, a contradiction.  \end{proof}

\begin{rem} \label{rem:homeo}
While proposition \ref{prop:isometry} does not hold for homeomorphisms of $M$, one may use Pardon's theorem that any continuous finite group action on a 3-manifold can be uniformly approximated by smooth actions \cite{Pardon} to show the following:  
{\em if $M$ admits a homeomorphism $h$ of sufficiently high order (here thought of as an action of a finite cyclic group), then $M$ is Seifert fibered and $h$ can be approximated by an isometry of a model geometric structure.}
\end{rem}

As a consequence of the above result, we can now prove Theorem \ref{torsion}.  As in the previous results, ``arbitrarily high order" can be replaced by a bound which depends on the topology of $M$.



\begin{proof}[Proof of \Cref{torsion}]
We claim that, if $f \in G_1$ and $g \in G_2$ have sufficiently large order then the action of the subgroup generated by $(f,1)$ and $(1, g)$ does not extend to $M$.  
Consider first the action of the finite order element $(f,1)$.   By \cite{Pardon}, the action of the group generated by $(f,1)$ can be uniformly approximated by an action of this cyclic group by diffeomorphisms.  Moreover, since the action is free on the boundary, we may in fact take such a diffeomorphism to be conjugate on the boundary to the original action of $(f,1)$.   Take such an action by diffeomorphisms, and let $M'$ be the prime summand of $M$ containing the torus boundary.  By the equivariant sphere theorem, there is a $(f,1)$-invariant sphere bounding the punctured $M'$, we may cone off the action of $(f,1)$ on this sphere to a glued in ball to produce an action of $(f,1)$ by homeomorphisms on $M'$ that agrees with the original action on the boundary.  
By Proposition \ref{prop:isometry}, if the order of $f$ is sufficiently large then the interior of $M'$ thus has a Seifert fibered structure.  Assume first for simplicity that this structure is unique up to isotopy.  

Proposition  \ref{prop:isometry} says that  there exists $n$, depending only on the topology of $M'$, such that $f^{n}$ acts on this by rotating the fibers of the unique fibration.    Let $a \subset  \partial M'$ be a regular fiber, it is invariant under (and rotated by) $f^{n}$.  
By Lemma \ref{lem:invariant_curve}, either $a$ is freely homotopic to the first $S^1$ factor of $S^1 \times S^1 = \partial M$, or there is an upper bound on the order of a rotation $r$ of the first $S^1$ factor, such that $r(a) = a$.   Now, if $f$ was originally chosen to have sufficiently high order, then taking $r = (f,1)^{n}$ will give a rotation of order higher than this bound.  Thus, we conclude that $a$ is freely homotopic to the first $S^1$ factor. 
Applying the same argument with $(1,g)$ in place of $(f,1)$ and using uniqueness of the Seifert structure shows that $a$ is also freely homotopic to the second $S^1$ factor, a contradiction.  

In the case where $M'$ has a non-unique fibering, the argument above simply shows that the two $S^1$ factors are fibers of different fiberings.  However, the same argument can be repeated with the finite order diffeomorphisms $(f,g)$, $(f^2, g)$, etc, implying that $M'$ would in fact need to have infinitely many different fiberings. The only possibility is that $M'$ is equal to the solid torus.

To treat the case where $M'$ is a solid torus, we argue as follows.  As before, we can apply the equivariant sphere theorem and get an action of $(f,1)$ on the solid torus agreeing with the original action on the boundary.   Now double the solid torus along its boundary gluing meridian to longitue to obtain a 3-sphere.  Realizing the glued-in torus as $D^2 \times S^1$, we may extend the finite order diffeomorphism $(f,1)$ of its boundary to a diffeomorphism of $D^2 \times S^1$ preserving each torus formed by the product of a circle of radius $r$ in the first factor with the second $S^1$ factor, and preserving the central $\{0\} \times S^1$.  One simply ``cones off" the original action, identifying concentric circles.   Thus, we get an action of $(f,1)$ on the sphere.  By Smith fixed point theory, any prime order orientation-preserving homeomorphism of the 3-sphere has fixed set equal to a (possibly knotted) topological circle.  Applying this to any nontrivial power of $f$ that has prime order, we get a contradiction because the action of $(f,1)$ on the torus boundary is free, and thus the fixed set for the double action cannot be a connected set.  \end{proof}

\begin{rem} 
The above proof would be simpler if one could choose the identification of the boundary of $M$ with $S^1 \times S^1$ using the structure of $M$, i.e. simply choose the $S^1$ factor to not be a regular fibering of the Seifert piece containing the boundary.  But this is not useful for our intended application to the bordism problem.   
\end{rem}

The next proposition gives the final ingredient in the proof of Theorem \ref{Thompson}.  
As in the introduction, we let $G'$ denote the smooth conjugate of Thompson's group in $\Diff^{\infty}(S^1)$. 
\begin{prop} \label{prop:solid torus} 
Let $G' \times G'$ act on $S^1 \times S^1$ via the standard action preserving each factor.  Suppose $M$ is a solid torus with boundary $S^1 \times S^1$ (we do not require either of the factors to be a disc-bounding curve).   
Then this action does not extend to an action by $C^1$ diffeomorphisms on $M$. 
\end{prop}

\begin{proof}
First, work with the non-smooth version of Thompson's group.  Note that the lift of such a homeomorphism to a $2^k$--fold cover of the circle is again an element of Thompson's group, and more generally, if $t$ is a finite order element of Thompsons group, then the lift of any element under the map $S^1 \to S^1/\langle t \rangle$ is again an element of Thompson's group, since its breakpoints are again at dyadic points and its slopes are dyadic.   Since the smooth version is a conjugate of this action, the same result on lifts is true for the smooth version.  

We give the proof first in the case where one of the $S^1$ factors bounds a disc, as this requires only a simple adaptation of work from \cite{MannNariman}.  
Assume for concreteness that the first $S^1$ factor bounds a disc.  We will work with the subgroup $G' \times \{1\}$ acting on $S^1 \times S^1$.  
Since Thompson's group $G'$ is perfect, every element is a product of commutators.  Let $r_2$ be an order 2 rigid rotation in $G'$, and write it as a product of commutators, $r_2 = [a_1, b_1] ...[a_k, b_k]$.
Choose lifts of $a_i$ and $b_i$ to the 4-fold cover of the circle.  These will all commute with the order 4 rotation.  Denote these by $A_i$ and $B_i$ respectively.  Let $r$ denote the product of commutators $r= [A_1, B_1]...[A_k, B_k]$.  Then $r$ is a lift of the order two element, and one may check that it has order 8.   Also, $r^2$ is the covering map, so commutes with each $A_i$ and each $B_i$.   

Now we may conclude the proof by directly applying an argument from \cite{MannNariman} that was inspired by a similar strategy used by Ghys \cite{Ghys}.  In outline, one first shows that $r$ has nonempty fixed set, as can be seen easily from lifting the action to the universal cover of the solid torus and this fixed set is equal to a circle embedded in the solid torus.  This uses only the fact that $r$ is a finite order element whose action on the boundary preserves the circles of the $S^1$ factor that bounds a disc, rotating each circle. Thus, taking a trivialization of a tubular neighborhood of this fixed set, the derivatives of $r^2$ at any such point may be taken to have the form $\left( \begin{smallmatrix} A & 0 \\ 0 & 1\end{smallmatrix} \right)$, where $A \in O(2)$ has order 4.  
Since $A_i$ and $B_i$ commute with $r^2$, the fixed set of $r^2$ (which is equal to the fixed set of $r$) is $A_i$ and $B_i$ invariant, and 
the derivatives of $A_i$ and $B_i$ commute with this linear map  $\left( \begin{smallmatrix} A & 0 \\ 0 & 1\end{smallmatrix} \right)$.  However,  it is easily checked that the centralizer of such a map is abelian.  Thus, we conclude that $r$ cannot be written as a product of commutators. 

Now we adapt this line of argument to the general case.  
Let $\alpha$ and $\beta$ be generators of the two $\pi_1(S^1)$ factors, respectively, and suppose that $\alpha^k\beta^m$ represents a simple curve which bounds a disc in $M$  (in particular $\alpha$ and $\beta$ are relatively prime).  In order to apply the same proof strategy as above, we need to find a finite order element $r$ of $G' \times G'$ which is conjugate to a rotation in the direction of $\alpha^k\beta^m$, and can be written as a product of commutators  $r= [A_1, B_1]...[A_k, B_k]$ in $G' \times G'$ such that some nontrivial power of $r$ has order at least 3 and commutes with each $A_i$ and $B_i$.  In this case, the end of the argument above applies verbatim.   Thus, the remainder of the proof is devoted to producing $r$, again by using tricks lifting to factors in each cover.  

Again, to do this, we may work with the non-smooth version of Thompson's group.   Let $g$ be a finite order element of order $|2m|$ and rotation number $\frac{1}{2m}$. Such an element may be constructed in the standard non-smooth Thompson's group by partitioning $S^1$ into $4m$ intervals of rational dyadic lengths and sending each one to the next in cyclic order if $m>0$ and reverse cyclic order if $m<0$.  We may additionally choose this partition so that $g^m$ is a rigid rotation of order 2.    Similarly, let $f$ be a finite order element of order $|2k|$ and rotation number $\frac{1}{2k}$, with $f^k$ a rigid rotation of order 2.   

Using the fact that $G'$ is perfect, write $g = [a_1, b_1]...[a_j, b_j]$ and $f = [c_1, d_1] ... [c_j, d_j]$, thus $(g, f)$ is the product of commutators $(g,f) = \prod_{i=1}^j [(a_i, c_i), (b_i, d_i)]$. Note that we may choose some of these commutators to be trivial in order to ensure that the expressions have the same length.  
Similar to the previous proof, consider now a degree $4$ cover $S^1 \to S^1$ that is a local isometry, and take lifts $A_i$ and $B_i$ of $a_i$ and $b_i$ to homeomorphisms of the cover; these will lie in Thompson's group, will commute with the deck transformation, and will satisfy that the product of commutators
\[ s:=  \prod_{i=1}^j[A_i, B_i]  \] 
has rotation number $\frac{1}{8m}$.  Also, since $s^m$ is a lift of the order 2 rigid rotation $g^m$, we have that $s^m$ is a rigid rotation of order $8$ with $s^{2m}$ equal to the deck transformation of the cover.  

In the same way, we may choose lifts $C_i$ and $D_i$ of $c_i$ and $d_i$ to a degree $4$ cover, and we have that 
\[ t:=  \prod_{i=1}^j[C_i, D_i]  \] 
has rotation number $\frac{1}{8k}$, that $t^k$ is a rigid rotation of order $8$, and $t^{2k}$ commutes with $C_i$ and $D_i$.  

Now we return to considering the product action of the product of two copies of Thompson's group on $S^1 \times S^1$.  
Let 
\[ r := (s, t) = \prod_{i=1}^j [(A_i, C_i), (B_i, D_i)] \] 
By construction, each of $(A_i, C_i)$ and $(B_i, D_i)$ commutes with $r^{2km}$, and $r$ is conjugate to a homeomorphism which preserves each $S^1$ factor, rotating the first by $\frac{1}{8m}$ and the second by $\frac{1}{8k}$; in other words, it has order $8km$ and rotates in the direction of $\alpha^k \beta^m$.  
Furthermore, $r^{2km}$ has order 4, which is the final property that we needed to show.   
One may now consider smooth conjugates of these elements and proceed as in the first case where one factor bounded a disc.  
\end{proof}



Since $G'$ has elements of arbitrarily high order, combining the above propositions yields \Cref{Thompson}.  

\begin{cor}
$G' \times G'$ is a finitely generated group with an action on $S^1 \times S^1$ which does not extend to an action by $C^1$ diffeomorphisms on any three manifold bounded by $S^1 \times S^1$.    With the exception of the case where the bounding manifolds is the solid torus, the action furthermore does not extend to an action by $C^0$ homeomorphisms.  
\end{cor}

\section{On powers of the Euler class}

Let $M$ be an irreducible $3$-manifold with boundary $\partial M\cong T^2$.   In this section we prove \Cref{main1}, showing that in the case where $M$ is not homeomorphic to the solid torus, there is a cohomological obstruction to extending groups acting on $\partial M$ to actions on $M$.
The advantage of cohomological obstruction in low dimensions is that it is insensitive to the regularity of the action, so we do not have to appeal to smoothing results to approximate $C^0$ actions by differentiable ones.  

Consider the map between classifying spaces induced by the restriction map
\[
\BDiff(M, \partial_0)\to \BDiff_0(T^2)\simeq \mathrm{B}T^2.
\]
\Cref{main1} is a consequence of the following proposition.
\begin{prop}\label{prop}
Let $M$ be an irreducible  $3$-manifold with boundary $\partial M\cong T^2$ such that it is not diffeomorphic to the solid torus. Then there exists an  integer $k$ such that the map induced by the restriction map
\[
H^k(\BDiff_0(T^2);\bQ)\to H^k(\BDiff(M, \partial_0); \bQ) 
\]
has a nontrivial kernel. 
\end{prop}
\begin{proof}[Proof of \Cref{main1}] 
If the action of $H$ on the boundary extends to a $C^0$-action on $M$, then we have a homotopy commutative diagram between classifying spaces
\begin{displaymath}
    \xymatrix @M=3pt {
          & \BH(M, \partial_0) \ar[d]^{r} \\
        \mathrm{B}H \ar[r]_{\hspace{-.7cm}\rho} \ar[ur]^\phi  & \BH_0(T^2). }
\end{displaymath}
 It is a well-known fact that in dimensions smaller than $4$, the inclusion of  diffeomorphism groups into homeomorphism groups is a weak homotopy equivalence (for dimension $3$ see \cite{cerf1961topologie} which is based on Hatcher's proof \cite{hatcher1983proof} of Smale's conjecture and for dimension $2$ see \cite{hamstrom1974homotopy}). Hence, the  the induced map 
\[
\rho^*: H^*(\BH_0(T^2);\bQ)\to H^*(\mathrm{B}H; \bQ),
\]
would be injective in all degrees by the hypothesis. But by \Cref{prop} the map $r^*$ has a nontrivial kernel and since $\rho=r\circ \phi$, it implies that $\rho^*$ also has a nontrivial kernel which is a contradiction.
\end{proof}
\begin{proof}[Proof of \Cref{prop}]Recall that $M$ is an irreducible $3$-manifold with a single torus boundary component. Since $M$ is not diffeomorphic to a solid torus, its boundary is incompressible. For such manifolds, the JSJ decomposition (\cite{MR551744, MR539411}) gives a canonical set of disjoint embedded incompressible tori and incompressible annuli (possibly empty), so that cutting $M$ along those tori and annuli gives a decomposition of $M$ into pieces that are either Seifert fibered, $I$-bundles over surfaces of negative Euler characteristic, or admit a hyperbolic structure on the interior (see also \cite[Section 4]{MR1486644}).
When $M$ has only torus boundary components - the case of interest to us - no annuli are needed in the decomposition (see \cite[Section 5]{neumann1996notes}) and the piece with torus boundary is either Seifert fibered or has hyperbolic interior.  

We consider three cases depending on whether the JSJ decomposition is trivial, and if trivial, depending on the structure of the piece with the torus boundary.

{\bf Case 1: $M\backslash T$ is hyperbolic.} Here and in what follows we use the notation $\text{Mod}(M)$ to denote the {\em mapping class group} $\pi_0(\Diff(M))$.  
 Then $\text{Mod}(M)$ is isomorphic to the group of isometries $\text{Isom}(M\backslash T)$ which is a finite group.



Recall that $\Diff(M, \partial_0)$ denotes the subgroup of $\Diff(M)$ that restricts to diffeomorphisms of the torus boundary $T^2$ that are isotopic to the identity and let $\text{Mod}(M, \partial_0)$ be the corresponding mapping class group. Since $\text{Mod}(M, \partial_0)$ is a subgroup of a finite group, it is also finite. 
 
 \begin{claim*}
 $\pi_i(\Diff(M,  \partial_0))=0$ for $i>0$.
 \end{claim*}
 \begin{proof}[Proof of the claim]There is a fibration 
 \begin{equation}\label{fib}
 \Diff(M,  \text{rel }\partial_0)\to \Diff(M,  \partial_0)\to \Diff_0(T^2),
 \end{equation}

 where $\Diff(M,  \text{rel }\partial_0)$ denotes the subgroup of $\Diff(M)$ that restricts to the identity on the torus boundary $T^2$. Since $M$ is Haken, Hatcher's theorem (\cite[Theorem 2]{hatcher1999spaces}) 
 implies that $\pi_i(\Diff(M,  \text{rel } \partial_0))=0$ for $i>0$. Hence, the long exact sequence of homotopy groups of the above fibration implies that $\pi_i(\Diff(M,  \text{rel } \partial_0))=\pi_i(\Diff(M,   \partial_0))$ for $i>1$. Hence, to finish the proof of the claim, we need to consider the case $i=1$.  In other words, we need to  show that the map 
 \[
 \bZ^2\to \text{Mod}(M,  \text{rel }\partial_0)
 \]
 in the long exact sequence of the homotopy groups of the fibration \ref{fib} is injective. 
 
The group $\pi_1(\Diff_0(T^2))=\bZ^2$ is generated by the $S^1$-actions along the meridian and the longitude of the torus boundary $T^2$ and their images in $\text{Mod}(M,  \text{rel }\partial_0)$ are realized by the Dehn twists in a collar neighborhood of $T^2$. Fix a base point $x\in T^2$, then the Dehn twists around $T^2$ act as inner automorphisms on $\pi_1(M, x)$. Given the hyperbolicity of $M$, the group $\pi_1(M, x)$ has no center. Therefore, the composition
 \[
 \bZ^2\to \text{Mod}(M,  \text{rel }\partial_0)\to \text{Aut}(\pi_1(M, x)),
 \]
 is injective. Hence, the first map has to be injective.
 \end{proof}
 The claim implies that $\Diff(M,  \partial_0)$ is homotopy equivalent to the finite group $\text{Mod}(M, \partial_0)$. Therefore, we have $\tilde{H}^*(\BDiff(M,  \partial_0);\bQ)=0$ which implies that the kernel of the map 
 \[
H^k(\BDiff_0(T^2);\bQ)\to H^k(\BDiff(M,  \partial_0); \bQ),
\]
is nontrivial. 


{\bf Case 2: $M$ is a Seifert fibered space.}  Let $M$ be a Seifert fibered manifold where $T^2$ is a torus boundary component. We may assume that the boundary of $M$ is union of non-singular fibers. Since $M$ is Haken and is not diffeomorphic to a solid torus, by Hatcher's theorem (\cite[Theorem 2]{hatcher1999spaces}), the identity component $\Diff_0(M, \partial_0)$ is either contractible or it has the homotopy type of $S^1$. 
First, let us assume that $\Diff_0(M,  \partial_0)\simeq *$. In this case, we have $\BDiff(M,\partial_0)\simeq \mathrm{B}\text{Mod}(M,  \partial_0)$. Since $H^*(\BDiff_0(T^2);\bQ)$ is a polynomial algebra over $\bQ$, if we show that $\text{Mod}(M,  \partial_0)$ is virtually cohomologically finite i.e. $H^*(\mathrm{B}\text{Mod}(M,  \partial_0);\bQ)$ vanishes above some degree, we can conclude that the map on rational cohomology induced by the restriction map 
\[
\mathrm{B}\text{Mod}(M,  \partial_0)\to \BDiff_0(T^2),
\]
has a nontrivial kernel. Hence it is enough to show that the mapping class group is virtually cohomologically finite.

The mapping class group of Seifert fibered spaces are well understood.  Except few exceptional cases, the Seifert fibered structure is unique and for these cases the mapping class group is isomorphic to the fiber-preserving mapping class group (\cite[Proposition 25.2 and Proposition 25.3]{MR551744}, see also \cite[Theorem 1]{preaux2012survey}). The only manifold with one torus boundary component among the exceptional cases is the solid torus which is excluded by the hypothesis. 

Hence, in our cases that mapping class group is isomorphic to the fiber-preserving mapping class group, and this group sits in a short exact sequence between ``vertical" and ``horizontal" mapping class elements as follows. Let $M$ be fibered over a surface $\Sigma$ with the projection $p:M\to \Sigma$ and let $S$ be the set of projections singular fibers to $\Sigma$. Let $\Diff^*(\Sigma, S)$ be the subgroup of $\Diff(\Sigma)$ that permute the points in $S$ with the same index of the corresponding singular fibers of Seifert fibered space structure. Let $\text{Mod}^*(\Sigma, S)$ be its group of connected components. Similar to \cite[Lemma 2.2]{MR1486644}, we have a short exact sequence
\[
1\to H_1(\Sigma, \partial \Sigma)\to \text{Mod}(M,  \partial_0)\to \text{Mod}^*(\Sigma, S)\to 1.
\]

Let $\Sigma_0$ be the surface obtained from $\Sigma$ by cutting out a neighborhood of $S$. Since $\text{Mod}^*(\Sigma, S)$ is finite index subgroup of $\text{Mod}(\Sigma_0)$ which is virtually cohomologically finite, so is $\text{Mod}^*(\Sigma, S)$. Moreover, the abelian group $H_1(\Sigma, \partial \Sigma)$ is also virtually cohomologically finite. Therefore, by \cite[Lemma 1.1]{MR1486644}, we have $\text{Mod}(M, \partial_0)$ is also virtually cohomologically finite. 

Now suppose $\Diff_0(M, \partial_0)\simeq S^1$. Now we have a homotopy commutative diagram
\begin{equation}
\begin{gathered}
\begin{tikzpicture}[node distance=1.5cm, auto]
  \node (A) {$\mathrm{B}S^1$};
  
  \node (C) [right of=A, node distance=3cm]{$\BDiff(M,\partial_0)$};
  \node (B) [ below of=C, node distance=1.5cm]{$\BDiff_0(T^2),$};
  \node (D) [right of=C, node distance=3cm]{$\mathrm{B}\text{Mod}(M, \partial_0)$};
  \draw [->] (A) to node {$\iota$} (C);
    \draw [<-] (B) to node {$r$} (C);
  \draw [->] (A) to node {$f$} (B);
  \draw [->] (C) to node {} (D);
 \end{tikzpicture}
 \end{gathered}
\end{equation}
where the horizontal maps give a fibration induced by $$\Diff_0(M,\partial_0)\to \Diff(M,\partial_0)\to \text{Mod}(M, \partial_0).$$

Note that $f^*: H^2(\BDiff_0(T^2);\bQ)\to H^2(\mathrm{B}S^1;\bQ)$ has a nontrivial kernel, say $x$ is a nontrivial element in the kernel. Therefore, $r^*(x)$ is in the kernel of $\iota^*$. Consider  the fibration $$\mathrm{B}S^1\to \BDiff(M,\partial_0)\to \mathrm{B}\text{Mod}(M, \partial_0).$$ Recall the filtration on the cohomology of the total space of a fibration $f\colon E\to B$ that gives rise to the Serre spectral sequence is induced by the pre-image of the skeleton filtration of $B$, and the filtration terms are given by $F_pH^n(E):= \ker(H^n(E)\to H^n(f^{-1}(\text{skl}_{p-1}B)))$. Hence, the first term of filtration on the second cohomology $H^2(\BDiff(M,\partial_0);\bQ)$ 
is given by $\ker(H^2(\BDiff(M,\partial_0))\xrightarrow{\iota} H^2(\mathrm{B}S^1))$. Since $r^*(x)$ lies in the kernel of $\iota^*$, it implies that $r^*(x)$ has a positive filtration in the Serre spectral sequence. Therefore, for some power $k$, we know that $r^*(x^k)$ has a Serre filtration beyond the $\bQ$-cohomological dimension of $\text{Mod}(M,\partial_0)$.  Thus, for some $k$, the class $r^*(x^k)$ has to vanish in rational cohomology. However, since $H^*(\BDiff_0(T^2);\bQ)$ has no nilpotent element, $x^k$ is a nontrivial element in the kernel of $r^*$.

{\bf Case 3: $M$ has a nontrivial JSJ decomposition.} In this case Hatcher's theorem \cite[Theorem 2]{hatcher1999spaces}
states that the identity component of $\Diff(M, \partial_0)$ is contractible. Let $T_P$ be the union of  tori in the JSJ decomposition that cuts out the unique piece $P$ containing the boundary torus $T^2$. 

Let $\Diff(M, T_P, \partial_0)$ denote the subgroup of $\Diff(M,\partial_0)$ that preserve $T_P$. From Hatcher's theorem \cite[Theorem 1]{hatcher1999spaces} on the homotopy type of spaces of embeddings of incompressible surfaces in a Haken manifold, for each component $T_0$ of $T_p$, we have  $\text{Emb}_{T_0}(T^2, M)\simeq T^2$ where $\text{Emb}_{T_0}(T^2, M)$ is the space of embeddings of tori isotopic to $T_0$.  As is also explained in \cite[Page 107]{MR1486644}, from the uniqueness of the JSJ decomposition and Hatcher's theorem, we conclude that the map 
\[
\text{Mod}(M, T_P, \partial_0)\to\text{Mod}(M, \partial_0),
\]
is an isomorphism.  Hence, given the homotopy commutative diagram

\begin{displaymath}
    \xymatrix @M=3pt {
          & \mathrm{B}\text{Mod}(P, \partial_0) \ar[d] \\
        \mathrm{B}\text{Mod}(M, T_P, \partial_0) \ar[r] \ar[ur]  & \BDiff_0(T^2), }
\end{displaymath}
we deduce that the map 
\[
r^*:H^*(\BDiff_0(T^2);\bQ)\to H^*(\mathrm{B}\text{Mod}(M, T_P, \partial_0);\bQ),
\]
factors through 
\[
r^*_P:H^*(\BDiff_0(T^2);\bQ)\to H^*(\mathrm{B}\text{Mod}(P, \partial_0);\bQ).
\]
But by previous cases $r^*_P$ has a nontrivial kernel. Therefore, $r^*$ also has a nontrivial kernel. 
 \end{proof}

\subsection*{Application to the extension or bordism problem}
Let $G \subset \Homeo_0(S^1)$ be Thompson's group.  (Altrnatively, one could work with its smooth conjugate $G'$, as described in the introduction).   
Here show there is a cohomological obstruction to extending the  product action of $G \times G$ on $S^1 \times S^1$ to an irreducible manifold $M$ with torus boundary.     

To do so, we recall what is known about the cohomology of $G$.  Ghys and Sergiescu (\cite[Theorem B]{MR896095}) used a theorem of Greenberg (\cite{MR906823}) to prove that there exists a map
 \[
 \mathrm{B}G\to \text{Map}(S^1, S^3)\hcoker S^1,
 \]
 which induces a homology isomorphism where $\text{Map}(S^1, S^3)\hcoker S^1$ is the homotopy quotient \footnote{For a topological group $G$ acting on a topological space $X$, the homotopy quotient is denoted by $X\hcoker G$ and is given by $X\times_G \mathrm{E}G$ where $\mathrm{E}G$ is a contractible space on which $G$ acts freely and properly discontinuously.} of the circle action on the space of loops on $S^3$. Hence, they conclude that 
 \[
 H^*(\mathrm{B}G;\bZ)\cong \bZ[\alpha, \chi]/\alpha\cdot \chi ,
 \]
 where $\chi$ is the Euler class induced by the inclusion $G\hookrightarrow \tH(S^1)$ and $\alpha$ is also a degree $2$ class which is a ``PL version" of Godbillon-Vey class. From this computation, what we need is $H_1(\mathrm{B}G;\bZ)=0$ so $G$ is perfect and the powers of the Euler class $\chi^k\in H^{2k}(\mathrm{B}G;\bZ)$ are non-zero for all $k$. Therefore, by \Cref{prop}, the action of $G\times G$ on the boundary does not extend to a $C^0$-action on $M$, unless $M$ is the solid torus.
 
\subsection*{Discussion: reducible case}
One approach to proving the same statement as in \Cref{prop} for a reducible manifold $M$ with a torus boundary component would be to generalize the solution of Kontsevich's conjecture by Hatcher and McCullough \cite{MR1486644} for reducible $3$-manifolds. For such $M$, in the previous paper (\cite[Theorem 1.2]{MannNariman}) we proved that when $M$ is not diffeomorphic to the solid torus, the map
 \[
H^2(\BDiff_0(T^2);\bQ)\to H^2(\BDiff_0(M, \partial_0); \bQ),
\]
has a nontrivial kernel. Let $x\in H^2(\BDiff_0(T^2);\bQ)$ be a nontrivial element in the kernel. Now, consider the homotopy commutative diagram 
\begin{equation}
\begin{gathered}
\begin{tikzpicture}[node distance=1.5cm, auto]
  \node (A) {$\BDiff_0(M, \partial_0)$};
  
  \node (C) [right of=A, node distance=3cm]{$\BDiff(M,\partial_0)$};
  \node (B) [ below of=C, node distance=1.5cm]{$\BDiff_0(T^2).$};
  \node (D) [right of=C, node distance=3cm]{$\mathrm{B}\text{Mod}(M, \partial_0)$};
  \draw [->] (A) to node {$\iota$} (C);
    \draw [<-] (B) to node {$r$} (C);
  \draw [->] (A) to node {$f$} (B);
  \draw [->] (C) to node {} (D);
 \end{tikzpicture}
 \end{gathered}
\end{equation}
Note that $r^*(x)$ has a positive Serre filtration in the Serre spectral sequence for the fibration $\BDiff_0(M, \partial_0)\to \BDiff(M, \partial_0)\to\mathrm{B}\text{Mod}(M, \partial_0)$.  If we knew that $\text{Mod}(M, \partial_0)$ were virtually cohomologically finite, similar to the case $2$ in the proof of \Cref{prop}, we could argue that for some integer $k$, the class $r^*(x^k)$ has to be zero which implies that the nontrivial class $x^k\in H^{2k}(\BDiff_0(T^2);\bQ)$ is in the kernel of the map 
\[
H^*(\BDiff_0(T^2);\bQ)\to H^*(\BDiff(M, \partial_0); \bQ).
\]
Therefore, this discussion will leave us with the following question.
\begin{problem}
Let $M$ be a reducible $3$-manifold such that $\partial M\cong T^2$. Is $\text{\textnormal{Mod}}(M, \partial_0)$ virtually cohomologically finite?
\end{problem}
Nonetheless, when $M$ is the connected sum of only two irreducible $3$-manifolds, one can prove the following
\begin{prop}
Let $P$ be an irreducible $3$-manifold such that $\partial P\cong T^2$ and $Q$ is a closed irreducible $3$-manifold. Let $M$ be $P\#Q$. Then the map 
\[
H^*(\BDiff_0(T^2);\bQ)\to H^*(\BDiff(M, \partial_0); \bQ),
\]
has a nontrivial kernel.
\end{prop} 
\begin{proof}
Let $S\subset M$ be a separating sphere. Since there is only separating sphere in $M$ up to isotopy, by a theorem of Hatcher \cite[Remark page 430]{MR624946} and Jahren \cite{MR2626194}, we know that $\Diff(M, S,  \partial_0)$ which is the subgroup of those diffeomorphisms that preserve the sphere $S$ setwise, is homotopy equivalent to $\Diff(M, \partial_0)$. Therefore, we have the following homotopy commutative diagram
\begin{displaymath}
    \xymatrix @M=3pt {
          & \BDiff(P\backslash D^3, \partial_0) \ar[d] \\
        \BDiff(M, S, \partial_0) \ar[r] \ar[ur]  & \BDiff_0(T^2).}
\end{displaymath}
Hence, it is enough to show that 
\[
H^*(\BDiff_0(T^2);\bQ)\to H^*(\BDiff(P\backslash D^3, \partial_0); \bQ),
\]
have a nontrivial kernel. But this follows exactly similar to the proof of \cite[Lemma 3.13]{MannNariman} by considering the cases whether $P$ is the solid torus and using \Cref{prop} for the irreducible case.
\end{proof}
Hence, we pose the general case as a question.
 \begin{problem}
 Let $M$ be a $3$-manifold with boundary $\partial M\cong T^2$ such that it is not diffeomorphic to the solid torus. Does the restriction map
\[
H^*(\BDiff_0(T^2);\bQ)\to H^*(\BDiff(M, \partial_0); \bQ),
\]
have a nontrivial kernel?
 \end{problem}

%
%
%
%
%
%

\bibliographystyle{alpha}
\bibliography{referenc}
\end{document}